\documentclass[final]{dmtcs-episciences}

\usepackage[utf8]{inputenc}
\usepackage{subfigure}

\usepackage{amsmath,amssymb,amsthm}
\usepackage{mathtools}
\usepackage{booktabs}
\usepackage{tikz}
\usetikzlibrary{calc}
\tikzstyle{vertex}=[circle, draw, fill, inner sep=0pt, minimum width=4pt]
\usepackage{mleftright}
\mleftright
\usepackage{microtype}

\newtheorem{theorem}{Theorem}

\newtheorem{claim}[theorem]{Claim}
\newtheorem{proposition}[theorem]{Proposition}
\newtheorem{conjecture}{Conjecture}
\newtheorem{problem}{Problem}

\usepackage[round]{natbib}

\hypersetup{hidelinks}

\newcommand{\eps}{\varepsilon}
\DeclareMathOperator{\unif}{Uniform}
\DeclareMathOperator{\deficit}{def}

\DeclarePairedDelimiter{\ceil}{\lceil}{\rceil}
\DeclarePairedDelimiter{\floor}{\lfloor}{\rfloor}

\author{Carla Groenland\affiliationmark{1}\thanks{Partially supported by the ERC Horizon 2020 project CRACKNP (grant no. 853234) and the Marie Skłodowska-Curie grant GRAPHCOSY (number 101063180).}
  \and Tom Johnston\affiliationmark{2,3}
  \and Jamie Radcliffe\affiliationmark{4}
  \and Alex Scott\affiliationmark{5}\thanks{Supported by EPSRC grant EP/V007327/1.}}
\title{Short reachability networks}
  
\affiliation{
  TU Delft, Delft, The Netherlands\\
  University of Bristol, Bristol, UK\\
  Heilbronn Institute for Mathematical Research, Bristol,  UK\\
  University of Nebraska--Lincoln, Lincoln, USA\\
  University of Oxford, Oxford, UK}
\begin{document}
\publicationdata{vol. 27:3}{2025}{14}{10.46298/dmtcs.12454}{2023-10-22; None}{2025-10-01}
\maketitle

\begin{abstract}
	We investigate the following generalisation of permutation networks.
	We say a sequence $T=(T_1,\dots,T_\ell)$ of transpositions in $S_n$ forms a $t$-reachability network if, for every choice of $t$ distinct points $x_1, \dots, x_t\in \{1,\dots,n\}$, there is a subsequence of $T$ whose composition maps $j$ to $x_j$ for every $1\leq j\leq t$. When $t=n$, any permutation in $S_n$ can be created and $T$ is a permutation network. Waksman [\emph{JACM}, 1968] showed that the shortest permutation networks have length about $n \log_2(n)$. In this paper, we investigate the shortest $t$-reachability networks for other values of $t$. Our main result settles the case of $t=2$: the shortest $2$-reachability network has length $\lceil 3n/2\rceil-2 $. For fixed $t \geq 3$, we give a simple randomised construction which shows that there exist $t$-reachability networks with $(2+o_t(1))n$ transpositions. We also study the effect of restricting to star-transpositions, i.e.~restricting all transpositions to have the form $(1, \cdot)$.
\end{abstract}

\section{Introduction}
Let $S_n$ be the symmetric group on $n$ elements, and write $(a,b)\in S_n$ for the transposition that swaps $a$ and $b$. Let $T_1, \dots, T_\ell$ be a sequence of transpositions with $T_i = (a_i, b_i)$ for $i = 1, \dots, \ell$. The sequence forms a \emph{permutation network} if for every choice of $\sigma \in S_n$, there is some subsequence $T_{i_1}, \dots, T_{i_{\ell'}}$ such that \[\sigma  = (a_{i_1}, b_{i_1}) \dotsb (a_{i_{\ell'}}, b_{i_{\ell'}}).\]
Equivalently, if counters are placed on each vertex of a complete graph on $n$ vertices, then a permutation network can reach any configuration of counters by following the switches prescribed by a suitable subsequence of $T_1,\dots,T_\ell$.

Since each transposition in a permutation network can only increase the number of possible final states by a factor of two, the network must have at least $\log_2(n!) = \sum_{i=1}^n \log_2(i)$ transpositions.
In the other direction, a construction of \cite{waksman1968permutation}
shows that permutation networks can be constructed with at most $\sum_{i=1}^n \ceil{\log_2(i)}$ transpositions when $n$ is a power of two (and this was independently shown by \cite{goldstein1967synthesis}). This bound was later extended 
to all values of $n$ by \cite{beauquier2002arbitrary} (although, according to \cite{kautz1968cellular}, this was known to Green much earlier).

\begin{theorem}[\cite{waksman1968permutation,goldstein1967synthesis,beauquier2002arbitrary}]
	\label{thm:waksman}
	There is a permutation network on $n$ elements using $\sum_{i=1}^n \ceil{\log_2(i)}$ transpositions.
\end{theorem}

Permutation networks are also called \emph{rearrangeable non-blocking networks} and have been extensively studied due to their usefulness in communication networks, cryptographic applications and distributed computing (see e.g.~\cite{dally2004principles,falk2021secure,holland2022efficient,smart2019distributing}).

While $\Theta(n \log n)$ transpositions are needed for a full permutation network, we can still hope to get good distributional properties with much shorter sequences of transpositions.
In this paper, we consider networks in which a given set of counters can be moved to any set of positions.  We say a sequence of transpositions $T_1, \dots, T_\ell\in S_n$ is a  \emph{$t$-reachability network} if for every choice of $t$ distinct ordered points $x_1, \dots, x_t\in [n]=\{1,\dots,n\}$, there is some subsequence $T_{i_1}, \dots, T_{i_{\ell'}}$ such that the composition $(a_{i_1}, b_{i_1}) \dotsb (a_{i_{\ell'}}, b_{i_{\ell'}})$ maps $j \mapsto x_j$ for $1\leq j \leq t$ (so an $n$-reachability network is exactly a permutation network).
Similar networks have been studied before (see e.g.~\cite{hromkovic2009size,lenfant1978parallel,oruc1994study}), but without any attempt to obtain tight bounds. Our aim in this paper is to provide exact or nearly exact bounds where possible.

Let us first consider the case of $1$-reachability where we want to distribute one counter across $n$ positions. By noting that there are at most $2^\ell$ possible subsequences of a sequence of length $\ell$, we immediately get the lower bound $\ell \geq \log_2(n)$, but this is far from optimal. Indeed, a transposition $(a,b)$ can only move the counter if either $a$ or $b$ is a position where there may already be a counter, and the number of `reachable positions' can only increase by one for each transposition.
This implies $n-1$ transpositions are needed, and it is easy to find a tight example, e.g.~$(1,2),(1,3),\dots,(1,n)$.

Our main result is an exact bound for $2$-reachability.
\begin{theorem}
	\label{thm:2reach}
	Let $n\geq 2$. The shortest $2$-reachability network on $n$ elements contains $\ceil{3n/2}-2$ transpositions.
\end{theorem}
We note that it is straightforward to find such a sequence, and the technical difficulty lies in proving the lower bound.

We also give a simple randomised construction which shows that one can achieve $t$-reachability using $(2 + o_t(1))n$ transpositions. Surprisingly, the coefficient of the leading term is independent of $t$.
\begin{theorem}
	\label{thm:treach}
	Let $t\geq 3$. There is a $t$-reachability network on $n$ elements of length at most $(2 + o_t(1))n$.
\end{theorem}

Reachability questions are closely related to the problem of generating uniform random permutations.
A \emph{lazy transposition} $T=(a, b, p)$ is the random permutation
\[
	T=\begin{cases}
		(a,b) & \text{ with probability }p, \\
		1     & \text{ otherwise},
	\end{cases}
\]
where $1\in S_n$ denotes the identity permutation.

Note that the composition of a sequence of lazy transpositions is also a random permutation.  A \emph{transposition shuffle} is a sequence of independent lazy transpositions $T_1, \dots, T_\ell$ such that $T_1\cdots T_\ell\sim \unif(S_n)$, the uniform distribution over $S_n$.
Let $U(n)$ denote the minimum $\ell$ for which there exists a transposition shuffle on $n$ elements of length $\ell$.
\cite{angel2018perfect} asked whether $U(n)=\binom{n}2$. This was disproved by \cite{uniformity} who gave the best-known upper bound $U(n)\leq \frac23 \binom{n}2+O(n\log n)$.
For a lower bound, note that, by ignoring the probabilities, any transposition shuffle gives a permutation network, and so $U(n)=\Omega(n\log n)$.\footnote{Conversely, any permutation network can be used to give a sequence which achieves every permutation with non-zero probability, but not necessarily the uniform distribution. It is important here that a transposition shuffle achieves the uniform distribution exactly: \cite{czumaj2015random} showed that there are sequences of lazy transpositions of length $O(n \log n)$ which are very close to uniform.}
This simple bound is still the best-known, and there remains a very substantial gap between the best-known upper and lower bounds.

Stronger results are known when the transpositions are restricted to specific families.
In the case when all transpositions are of the form $(i, i+1)$,
\cite{angel2018perfect} showed that $\binom{n}2$ is best possible. However, with this restriction there is also a lower bound of $\binom{n}2$ from the `reachability' point of view: $\binom{n}{2}$ transpositions are needed to reach the `reverse permutation' that maps $i$ to $(n+1)-i$. It is an interesting open problem to prove lower bounds that are better than $\Omega(n\log n)$ in the general case.

\smallskip

Another natural family of transpositions are the \emph{star transpositions}, that is all
transpositions of the form $(1, \cdot)$ (so corresponding to the edges of a star $K_{1,n-1}$).
Observe that restricting to star transpositions does not affect the order of magnitude
of the number of transpositions needed for uniformity as any lazy transposition $(i,j,p)$ can be simulated by the three star transpositions $(1,i,1), (1,j,p), (1,i,1)$.

A modification of Theorem \ref{thm:2reach} shows that  $2$-reachability remains possible with around $3n/2$ transpositions, even when restricting to star transpositions.

\begin{theorem}
	\label{thm:2reachstar}
	For $n \geq 3$, the shortest $2$-reachability network on $n$ elements in which each transposition is a star transposition contains $\ceil{3(n-1)/2}$ transpositions.
\end{theorem}

Just as permutation networks strengthen to transposition shuffles, there is a natural strengthening of $t$-reachability to the stochastic setting.
A sequence of lazy transpositions $T_1,\dots,T_\ell$ is a \textit{$t$-uniformity network} if the probability that the composition $T_1\dots T_\ell$ maps the tuple $(1,\dots,t)$ to $(x_1,\dots,x_t)$ is the same for every tuple $(x_1,\dots,x_t)$ of $t$ distinct elements from $[n]$. Of course, any $t$-uniformity network gives rise to a  $t$-reachability network.

Using the restrictive nature of the transpositions, we are able to show that achieving $2$-uniformity is strictly harder than achieving $2$-reachability.

\begin{theorem}
	\label{thm:2unifstar}
	There exists $C>0$ such that any $2$-uniformity network on $n$ elements in which each transposition is a star transposition has at least $1.6 n - C$ transpositions.
\end{theorem}
This is the only setting in which a separation between reachability and uniformity is known.\footnote{After a preliminary version of this paper was put on arXiv, Conjecture \ref{cj9} was proven by \cite{janzer2022partial}, and this shows separation between 2-reachability and 2-uniformity even without the restriction to star transpositions.}
In the setting of star transpositions, we also provide a lower bound on the length of $t$-reachability networks (see Proposition \ref{prop:star-reach}).

The remainder of the paper is organised as follows. We settle the case of 2-reachability (Theorem \ref{thm:2reach}) in Section \ref{sec:reachability} and prove Theorem \ref{thm:treach} in Section \ref{sec:treach}. In Section \ref{sec:improved-star} we study star transpositions, providing a separation between 2-reachability and 2-uniformity (Theorem \ref{thm:2reachstar} and Theorem \ref{thm:2unifstar}). We finish with some open problems in Section \ref{sec:discussion}.

\section{Exact bound for 2-reachability}
\label{sec:reachability}
We first prove Theorem \ref{thm:2reach} which asserts that the smallest  $2$-reachability network on $n$ elements has length $\lceil 3n/2\rceil-2$.

\begin{proof}[of Theorem \ref{thm:2reach}]
	We first give an upper bound construction. Let $n\geq 2$ be given. Our sequence of transpositions starts with the following transpositions, in order.
	\begin{enumerate}
		\item $(1,2)$.
		\item $(1,x)$ for all odd $3 \leq x \leq n$.
		\item $(2,y)$ for all even $4 \leq y \leq n$.
		\item $(x, x+1)$ for all odd $3 \leq x \leq n -1$.
	\end{enumerate}
	If $n$ is odd, we also add the transposition $(1,2)$ at the end, which is needed to reach the position $(1,n)$. This defines a sequence of transpositions of length $\ceil{3n/2}-2$, and it is straightforward to check that this sequence forms a 2-reachability network.

	\begin{figure}
		\centering
		\begin{tikzpicture}
			\draw node[vertex](1) at (0,0) {};
			\draw node[vertex](2) at (3, 0) {};
			\draw node[vertex](3) at (135:1) {};
			\draw node[vertex](5) at (105:1) {};
			\draw node[vertex](7) at (75:1) {};
			\draw node[vertex](9) at (45:1) {};
			\draw node[vertex](4) at ($(3,0) + (120:1)$) {};
			\draw node[vertex](6) at ($(3,0) + (90:1)$) {};
			\draw node[vertex](8) at ($(3,0) + (60:1)$) {};

			\draw (1) -- (2);
			\draw (1) -- (3);
			\draw (1) -- (5);
			\draw (1) -- (7);
			\draw (1) -- (9);
			\draw (2) -- (4);
			\draw (2) -- (6);
			\draw (2) -- (8);
			\draw[red] (3) to[bend left = 45] (4);
			\draw[red] (5) to[bend left = 45] (6);
			\draw[red] (7) to[bend left = 45] (8);
			\draw[red] (1) to[bend left] (2);

			\draw node[below = 0.2cm] at (1) {1};
			\draw node[below = 0.2cm] at (2) {2};
			\draw node[above right] at (9) {9};
		\end{tikzpicture}
		\caption[The multigraph corresponding to the minimum 2-reachable sequence given in the proof of Theorem \ref{thm:2reach}.]{The multigraph corresponding to the minimum 2-reachable sequence given in the proof of Theorem \ref{thm:2reach}. All vertices except $1$ and $9$ have deficiency 0.}	\label{fig:reach_example_1}
	\end{figure}
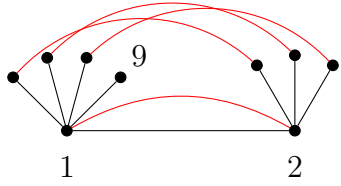

	The majority of the work in this proof is in showing the lower bound. Suppose that $\sigma=(\sigma_1,\dots,\sigma_\ell)$ is a shortest  $2$-reachability network on $n\geq 2$ elements. Given two (distinct) counters on positions $1$ and $2$, each subsequence of $\sigma$ defines a permutation that moves the counters to new (distinct) positions $x,y\in \{1,\dots,n\}$. By the definition of $2$-reachability, it must be possible to reach any such pair of distinct positions $(x,y)$.

	We will consider a process on an auxiliary multigraph $G$ with vertex set $\{1, \dots, n\}$. We initialise $G$ as the empty graph and process the transpositions in order, adding an edge for each. We maintain a set of active vertices, which starts as the set $\{1,2\}$. Any other vertex becomes active when it first appears in a transposition with an active vertex. Note that no counter can be sitting on an inactive vertex, and (by minimality) no transposition joins two inactive vertices.
	When we process a transposition $(a,b)$, we add an edge $ab$ to $G$. We colour the edge black if one of $a$ or $b$ was inactive, and red otherwise (with one exception detailed later).
    Note that the process begins by growing two trees of black edges from the vertices 1 and 2, which we denote $T_1$ and $T_2$ respectively, and that throughout the process, the black edges form a forest consisting of at most two trees.
    An example of the final multigraph (for the above upper bound construction when $n = 9$) is drawn in Figure \ref{fig:reach_example_1}.

	We now modify our sequence to obtain a sequence with a nicer form (and at most the same length).
	Let $c_1$ and $c_2$ denote the counters that start on $1$ and $2$ respectively, and let $(a,b)$ where $a \in V(T_1)$ and $b \in V(T_2)$ be the transposition at the first time the two trees meet. Before this time, the counter $c_1$ is contained in $V(T_1)$ and the counter $c_2$ is contained in $V(T_2)$, so we may replace the initial sequence of transpositions, up to but not including $(a,b)$, by the sequence which starts with $(1, v)$ for all $v \in V(T_1) \setminus \{1\}$ and then continues $(1,u)$ for all $u \in V(T_2) \setminus \{2\}$. This has not increased the number of transpositions and the counters can reach any position $(x,y) \in V(T_1) \times V(T_2)$, so the new sequence can reach every pair that the original sequence could. Observe also that there is no longer any distinction at this point between the positions in $V(T_1)$ and, separately, between the positions in $V(T_2)$. Hence, by switching all instances of $1$ and $a$ and all instances of $2$ and $b$, we can assume that $(a,b) = (1,2)$. 
    Although both $1$ and $2$ are active vertices, we colour this particular copy of the edge $(1,2)$ black, so the black edges will form a spanning tree which we view as rooted at the pair of vertices $\{1,2\}$.

	We now prove that the sum of the red degrees (counted at the end of the process) is at least $n-2$, which shows that the number of edges in $G$ (and transpositions in our sequence) is at least $n-1+\ceil{(n-2)/2}=\ceil{3n/2} - 2$, as desired. We first need some more definitions. The \emph{subtree} $T_v$ rooted at $v$ is the set of vertices $u$ (including $v$) such that the unique path from $u$ to $\{1,2\}$ in the black tree passes through $v$. Any vertex in the subtree rooted at $v$ is said to be \emph{above} $v$.
	The \emph{deficit} of a vertex $v$ is defined as
	\begin{equation*}
		\deficit(v)=|V(T_v)|-\sum_{u\in V(T_v)} d_R(u),
	\end{equation*}
	where $d_R(u)$ denotes the number of red edges incident to $u$. Note that the red edges may join vertices of $T_v$ to vertices outside $T_v$, so may not sit inside the tree. In order to bound the sum of the red degrees, we will inductively bound the deficit of the vertices.

	For each vertex $v\not\in \{1,2\}$ there is a unique vertex adjacent to $v$ on the unique path in the black tree from $\{1,2\}$ to $v$. We call this the \emph{parent} of $v$ and we denote it by $p(v)$ (which could equal $1$ or $2$). The \emph{children} $\mathcal{C}(v)$ of a vertex $v$ are the vertices $u$ for which $p(u) = v$. We may rewrite the formula for the deficit of $v$ in the following inductive manner.
	\begin{equation*}
		\deficit(v) = 1 - d_{R}(v) + \sum_{u \in \mathcal{C}(v)} \deficit(u).
	\end{equation*}
	Note that the first edge that can carry a counter to $v \not\in \{1,2\}$ is the black edge to its parent; if $v$ is not incident to a red edge, then this is the only way to get a counter to $v$.

	We show the following claim inductively, which we will extend to $v \in \{1,2\}$ to finish the proof.

	\begin{claim}\label{claim:6}
		Every vertex $v \not \in \{1,2\}$ has deficit at most 1. Moreover, if the deficit of $v$ is equal to 1, then there is a vertex $\ell$ such that the only way for a counter to reach $\ell$ is to enter $v$ using the black edge $p(v) v$ and then to follow the path of black edges from $v$ to $\ell$.
	\end{claim}
	\begin{proof}
		We prove the claim by induction on the height of the subtree rooted at $v \not\in \{1,2\}$. If the height is $0$, that is, $v$ is a leaf of the black tree, then the deficit is at most $1$ and equals $1$ if and only if $v$ can only be reached via the black edge from its parent $p(v)$.

		Now suppose that the claim has been shown up to height $h\geq 0$, and suppose that the subtree rooted at $v$ has height $h+1$.
		We first show that there are at least as many red edges incident to $v$ as there are children of $v$ with deficit 1, which implies that the deficit of $v$ is
		at most 1.

		If a vertex $u$ of deficit $1$ is above $v$ and connected to $v$ via a black edge, then by induction there exists a vertex $\ell(u)$ in the subtree rooted at $u$ which can only be reached using black edges from $v$. Let $vz$ be the last black edge from $v$ to a child of $v$ with deficit $1$. Then it is not possible to get counters into both $v$ and $\ell(z)$ unless there is a red edge
		incident with $v$
		after $(v,z)$. Similarly, if $v$ has black edges to vertices $x, y \in \mathcal{C}(v)$ that have deficit 1, then there has to be a red edge incident with $v$, produced by some transposition that occurs between the transpositions $(v,x)$ and $(v,y)$ else there is no way to reach $(\ell(x), \ell(y))$.

		Suppose now that $v$ has deficit exactly 1. We need to find a vertex $\ell$ which can only be reached using a path of black edges from the parent of $v$. If none of the children of $v$ have deficit 1, then the deficit of $v$ is $1 - d_R(v)$ and there are no red edges incident to $v$. Hence, we can take $\ell = v$. Otherwise, let $x$ be the first neighbour above $v$ which has deficit 1. Since, $v$ has deficit exactly 1, our argument above shows that there cannot be a red edge incident to $v$ which corresponds to a transposition before $(v, x)$. Hence, the only way of getting a counter into $\ell(x)$ is using the black edge to $v$, the edge $vx$ and the path of black edges from $x$ to $\ell(x)$, and so we may take $\ell = \ell(x)$.
		This completes the proof of Claim \ref{claim:6}.
	\end{proof}

	We now show that the vertices $1$ and $2$ have deficit at most $1$ as well. The difficulty in this case is that it is possible to put a counter on $1$ (say) using the black edge $(1,2)$. This means that we can put a counter on $\ell(x)$ and $\ell(y)$ without requiring a red edge between them. However, this does not take into account the fact that the counters are distinguishable,	and we only need to modify the argument from the claim slightly to make use of this.

	Let the counters starting in positions 1 and 2 be $c_1$ and $c_2$ respectively, and suppose that there is a child $x$ of $1$ which has deficit 1.
	Then there is a vertex $\ell(x)$ such that the only way for a counter to reach $\ell(x)$ is to follow the path of black edges from $1$ to $x$ to $\ell(x)$.
	If the transposition $(1,x)$ occurs before the black edge $(1,2)$, then there is no way to move the counter $c_2$ to $\ell(x)$, a contradiction. The black edge $(1,2)$ is the only way to put a counter on $1$ without using a red edge, so after this transposition we can apply a proof similar to that of the claim above to see that the deficit of 1 is at most 1. Indeed, if $x$ and $y$ are children of $1$ with deficit 1, then there has to be some transposition that can place a counter on $1$ in between $(1,x)$ and $(1,y)$ in order to reach $(\ell(x), \ell(y))$. We have just argued that this is not the black transposition $(1,2)$, and so it must be a red transposition. Likewise, there must be a red transposition after the last black edge to a vertex $z$ of deficit 1 in order to end with the counters in $1$ and $\ell(z)$.

	Since every vertex is in either the subtree rooted at 1 or the subtree rooted at 2 and the deficit of these two subtrees is at most $1$, the sum of the red degrees must be at least $n - 2$. Hence, there are at least $\ceil{(n-2)/2}$ red edges and at least $\ceil{3n/2}-2$ transpositions in the sequence.
\end{proof}

\section{Upper bound for \texorpdfstring{$t$}{t}-reachability}
\label{sec:treach}
We now give a probabilistic construction which shows that there are $t$-reachability networks on $n$ elements of length $(2+o_t(1))n$. This is much smaller than the best known constructions for uniformity and, surprisingly, the coefficient of the leading term is bounded. We remark that all the transpositions used in this construction are star transpositions.

\begin{proof}[of Theorem \ref{thm:treach}]
	Let $t\geq 3$ be a natural number and choose
	$\eps \in(0,1/(t+1))$.  Let $n>t$ and set $L = \floor{n^{1-\eps}}$. 	Recall that the counters start in positions $\{1, \dots, t\}$.

	Let $G$ be a random bipartite graph with vertices $A  \cup B = \{a_{t+1}, \dots, a_n\} \cup \{b_1, \dots, b_L\}$ constructed by adding two uniformly random edges from $a_j$ to $\{b_1, \dots, b_L\}$ for each $j$. The random transposition sequence begins with $L$ phases. In the $i$th phase, we add the transpositions $(1, j)$ for each $j \in \{2, \dots, t\}$ followed by $(1, j)$ for the $j \in \{t + 1,\dots n\}$ such that $a_j$ is adjacent to $b_i$. Finally, the sequence ends with a $t$-reachable network over the first $t$ positions.

	Fix positions $x_1,\dots,x_t\in \{1,\dots,n\}$. We need to find a subsequence that puts our $t$ counters into those positions. Let the $x_i$ which are in $\{t+1,\dots,n\}$ be $x_{i_1}, \dots, x_{i_m} = y_1, \dots, y_m$.

	Suppose that there is a matching $\{a_{y_j} b_{s_j} : j \in [m]\}$ in $G$ containing the vertices $a_{y_1}, \dots,  a_{y_m}$. For  $j \in [t]$, we can put the $j$th counter in position $y_j = x_{i_j}$ during phase $s_j$ using the transpositions $(1, j)$ and $(1, y_j)$. The remaining counters can easily be positioned using the $t$-reachable sequence at the end.

	It remains to show that there is some choice for $G$ such that every set of $s\leq t$ vertices from $A$ is in a matching.
	The probability that a given set of $s$ vertices from $A$ has at most $s-1$ neighbours in $B$ is at most
	\[ \binom{L}{s-1} \left( \frac{s-1}{L}\right)^{2s} \leq \frac{(s-1)^{2s}}{(s-1)!}L^{-(1+s)}. \]
	Hence, by the union bound, the probability that there is a set $A'\subseteq A$ of size at most $t$ and with at most $|A'|-1$ neighbours is at most
	\[ \sum_{s =3}^{t} \binom{n}{s} \frac{(s-1)^{2s}}{(s-1)!}L^{-(1+s)} \leq
		\sum_{s =3}^{t}  \frac{(s-1)^{2s}}{s!(s-1)!} \frac{n^s}{L^{s+1}}
		= O_t \left( n^{-1 + (t+1) \varepsilon} \right). \]
	Since $\eps < 1/(t+1)$, this probability is less than 1 for $n$ sufficiently large. This implies there exists a suitable choice for the graph $G$ (for sufficiently large $n$).

	Note that for $j \in \{t+1,\dots,n\}$, the  transposition $(1,j)$ appears exactly twice in the sequence since the vertex $a_j$ has degree exactly 2. We can create a $t$-reachable network  on the first $t$ positions using $O(t\log t)$ star transpositions.  Hence, there exists a $t$-reachability network using at most
	\[ (t-1)L + 2(n- t) + O(t\log t) = 2n + o_t(n) \]
	transpositions.
\end{proof}

\section{Star transpositions}
\label{sec:improved-star}

We now turn our attention to the setting where all transpositions are of the form $(1,\cdot)$ and prove Theorem \ref{thm:2reachstar}, which shows that restricting to star transpositions leads to only a small difference in the number of transpositions needed for $2$-reachability. In fact, there is no difference when $n$ is odd, and they differ by only 1 when $n$ is even.

Let us first consider the upper bound. Consider the following sequence: start with the transposition $(1,2)$ so that the two counters are indistinguishable, then sweep through the even numbers and potentially load one of them with a counter. Finally, sweep through every position. That is,
\begin{enumerate}
	\item $(1,2)$,
	\item $(1,4), (1,6), \dots, (1, 2\floor{n/2})$,
	\item $(1,2),(1,3), (1,4), \dots, (1,n)$.
\end{enumerate}
It is easy to see that the counters can reach $(x,y)$ when $x$ and $y$ are not both odd: simply place one counter in an even position in the first sweep and place the other counter in the second sweep. To put a counter on $x$ and $y$ when both are odd and $1 < x < y$, we load $(x+1)$ in the first sweep and use this to ``reload'' $1$ before the transposition $(1,y)$. When $x = 1$, this sequence works when $n$ is even and we can use $(1,n)$ to leave a counter on 1, but it breaks down when $n$ is odd. This could be fixed by appending the transposition $(1,2)$ say, but it is possible to do slightly better with the following ``twisted" sequence. Let $n = 2m +1$. The following sequence of transpositions is 2-reachable.

\begin{enumerate}
	\item $(1,2)$,
	\item $(1,4), (1,6), \dots, (1,2m)$,
	\item $(1,3),(1,2), (1,5), (1,4), \dots, (1,2m+1), (1,2m)$.
\end{enumerate}

\begin{proof}[of Theorem \ref{thm:2reachstar}]
	The upper bound is given in the discussion above, so we only need to prove a matching lower bound. For this we proceed much as in the proof of Theorem \ref{thm:2reach}. We colour the transposition $(1,x)$ black if it is the only occurrence, red if it has occurred previously and blue otherwise (i.e.~if it is a first occurrence, but will occur later as well). There are $n-1$ transpositions that are either black or blue, and it suffices to show that there are at least $(n-1)/2$ red transpositions.

	We claim that the number of red transpositions is at least the number of black transpositions. Let the two counters be $c_1$ and $c_2$. Suppose $(1,x)$ is the last black transposition where $x \geq 3$. Then there needs to be a red transposition after $(1,x)$ in order to leave the counters $c_1$ and $c_2$ in $1$ and $x$ respectively. Similarly, in between any pair of black transpositions $(1,x)$ and $(1,y)$ where $x,y \geq 3$ there has to be a red transposition in order to leave the counters $c_1$ and $c_2$ in $y$ and $x$ respectively. This shows the claim when $(1,2)$ is used multiple times, and there is no black $(1,2)$.

	If $(1,2)$ is used exactly once, then it must be the first black transposition. Indeed, if $(1,x)$ is a black transposition where $x \geq 3$, then there is no way to put the counter $c_2$ in $x$ unless $(1,2)$ has already occurred. If $(1,2)$ were to be the only black transposition, then $(1,x)$ must appear twice for every $x \geq 3$ and there are $2(n-2) + 1$ transpositions. This is at least $\ceil{3(n-1)/2}$ for $n \geq 3$ and we would therefore be done. Hence, if $(1,2)$ is used exactly once, it must be the first but not the last black transposition. The above argument shows there is a red transposition after the last black transposition and in between every pair of black transpositions $(1,x)$ and $(1,y)$ where $x,y \geq 3$. It is enough to show there is a red transposition between $(1,2)$ and the first black transposition $(1,x)$ where $x \geq 3$, and this is easy to argue: there must be a red transposition between $(1,2)$ and $(1,x)$ in order to leave $c_1$ in $x$ and $c_2$ in $1$.

	By definition, the number of red transpositions is at least the number of blue transpositions. Since there is a total of $n-1$ black and blue transpositions, there must be at least $(n-1)/2$ of one of them and there are at least $(n-1)/2$ red transpositions.
\end{proof}

The lower bound for 2-reachability above immediately gives a lower bound for the more difficult problem of constructing a 2-uniformity network with star transpositions but, using the restrictive nature of the transpositions, we can show that a 2-uniformity network has length at least $1.6 n - C$, a constant factor higher. This confirms, for this specific case, that the problem of uniformity is
strictly more difficult than reachability.

\begin{proof}[of Theorem \ref{thm:2unifstar}]
	We use a discharging argument to show that, after disregarding a constant number of the transpositions, the value 1.6 is a lower bound on the average number of times that a star transposition is used.

	Let $T=(T_1,\dots,T_\ell)$ be a sequence of star transpositions and suppose that $T$ is a 2-uniformity network on $n$ elements.
	We assign each transposition $(1,a)$ a weight equal to the number of times $(1, a)$ is used in $T$.
	Let $\sigma_1, \dots, \sigma_m$ be the transpositions of $T$ which are used exactly once and are of the form $(1, x)$ for $x \geq 3$. We transfer weight from the transpositions which are used multiple times to the transpositions which are used exactly once according to the following rules.
	\begin{itemize}
		\item If the last appearance of $(1,a)$ and $(1,b)$ is between $\sigma_i$ and $\sigma_{i+1}$, then they each transfer $0.3$ to $\sigma_i$ and $0.1$ to $\sigma_{i+1}$.
		\item If there is only one transposition which appears for the last time between $\sigma_i$ and $\sigma_{i+1}$, it transfers $0.4$ to $\sigma_i$.
		\item Each transposition which appears between $\sigma_i$ and $\sigma_{i+1}$ for neither the first time nor the last time, transfers $0.6$ to $\sigma_i$ and $0.2$ to $\sigma_{i+1}$.
	\end{itemize}
	If $(1,a)$ is used exactly twice, then it transfers $0.4$ and ends with weight $1.6$. A transposition used more than twice transfers $0.8(i-2) + 0.4$ so ends with weight $1.6 + 0.2(i-2)$. Hence, we only need to show that all but a constant number of the transpositions that are used once (the $\sigma_i$) end with weight at least $1.6$.

	Let $\sigma_{i-1} = (1,a)$ and $\sigma_{i} = (1,b)$. In order to end with counters in both $a$ and $b$, there must a transposition between $\sigma_{i-1}$ and $\sigma_{i}$ which can place a counter in position 1. In particular, there is either a transposition which is not appearing for the first time, or $(1,2)$ appearing for the first time. This immediately shows that all but one of $\sigma_1, \dots, \sigma_{m-1}$ have weight at least 1.4. We claim that in between either $\sigma_{i-1}$ and $\sigma_{i}$ or between $\sigma_{i}$ and $\sigma_{i+1}$ one of the following must hold:
	\begin{enumerate}
		\item there are at least two transpositions appearing for the last time,
		\item there is a transposition appearing for neither the first nor the last time,
		\item there is the transposition $(1,2)$ appearing for the first time.
	\end{enumerate}
	If one of the first two cases occurs, then $\sigma_i$ ends with weight at least 1.6, while the last case can only occur twice. This analysis does not apply to $\sigma_1$ and $\sigma_m$, and the number of $\sigma_i$ which end up with weight less than 1.6 is at most four.

	We now prove the above claim. Let $\sigma_{i-1} = (1,a)$, $\sigma_{i} = (1,b)$ and $\sigma_{i+1} = (1,c)$ and assume that none of the conditions above hold. In between $\sigma_{i-1}$ and $\sigma_i$, there can be only transpositions used for the first time and a single transposition $(1, \ell)$ used for the last time. We may write the sequence as

	\begin{equation}\label{eqn:seq}
		(1,a),(1,f_1),\dots,(1,f_k),(1,\ell),(1,f_{k+1}),\dots,(1,f_s),(1,b)
	\end{equation}
	where the $f_j$ are distinct, not equal to 2 and are appearing for the first time.

	The only way to end the sequence of lazy transpositions with the counters in $\{a,b\}$, $\{a, \ell\}$ or $\{b,\ell\}$ is to start the sequence (\ref{eqn:seq}) with the two counters in $\{1, \ell\}$. Let $p_j$ be the probability associated with the transposition $(1,j)$ in (\ref{eqn:seq}). The probability that this sequence ends with counters in $\{a,b\}$ and $\{a, \ell\}$ must be equal, so
	\[p_a(1-p_\ell) = p_a p_\ell (1-p_{f_{k+1}})\cdots (1-p_{f_s})p_b. \]
	In particular, if $q = (1-p_{f_{k+1}})\cdots (1-p_{f_s})$, then $1-p_\ell=p_\ell q p_b$ which implies $p_\ell = 1/(qp_b+1)\geq 1/2$.
	We now claim that the probability that a counter ends in $\ell$ is strictly higher than the probability that a counter ends in $b$ unless $q p_b = 1$.  We condition on which of $1$ and $\ell$ contain a counter before the transposition $(1, \ell)$ and calculate the probability that a counter ends in each of the positions $\ell$ and $b$.
	\begin{center}
		\begin{tabular}{ccccc}
        \toprule
            & \multicolumn{4}{c}{Counter positions} \\
            \cmidrule(lr){2-5}
			Event   & \{1,$\ell$\} & \{$\ell$\}    & \{1\}            & $\emptyset$ \\
            \midrule
			Counter ends in $\ell$ & 1            & $1-p_\ell$    & $p_\ell$         & 0           \\
			Counter ends in $b$    & $qp_b$       & $p_\ell qp_b$ & $(1-p_\ell)qp_b$ & 0           \\
            \bottomrule
		\end{tabular}
        \medskip
	\end{center}
	In every case, the probability that a counter ends in $\ell$ is at least the probability that a counter ends in $b$, and the first one is strict unless $qp_b = 1$.

	Now we consider the transpositions between $(1,b)$ and $(1,c)$:
	\[
		(1,b),(1,f_1'),\dots,(1,f_k'),(1,\ell'),(1,f_{k+1}'),\dots,(1,f_s'),(1,c).
	\]
	Since the transposition $(1,b)$ fires with probability 1, there is no way for the counters to end in both $\ell'$ and $c$, giving a contradiction and proving the claim.
\end{proof}

We have seen in Theorem \ref{thm:2reachstar} that the shortest 2-reachability network using only star transpositions has length approximately $3n/2$. It is also possible to construct a 3-reachable network using $5n/3 + C$ such transpositions. Using a similar argument to the proof of Theorem \ref{thm:2unifstar}, we offer the following lower bound for arbitrary $t$ when using star transpositions, which is tight up to additive constants when $t = 2, 3$.
\begin{proposition}
	\label{prop:star-reach}
	For each $t\geq 1$, there is a constant $C_t$ such that any $t$-reachability network on $n$ elements using only star transpositions  has length at least $(2-1/t)n-C_t$.
\end{proposition}
\begin{proof}
	We will again apply a discharging argument.
	Let $T$ be a given $t$-reachability network consisting of star transpositions.
	As before, we assign the transposition $(1,a)$ a weight equal to the number of times $(1,a)$ is used
	and let $\sigma_1,\dots,\sigma_m$ be the subsequence of transpositions that are used exactly once and are of the form $(1,x)$ for $x \not \in [t]$. We transfer weight from a transposition used multiple times to the $\sigma_i$ using the following simple rules.
	\begin{itemize}
		\item For each transposition between $\sigma_i$ and $\sigma_{i+1}$ used for the last time (but not the first), transfer $1/t$ to $\sigma_i$.
		\item For each transposition between $\sigma_i$ and $\sigma_{i+1}$ used for neither the first time nor the last time, transfer 1 to $\sigma_i$.
	\end{itemize}
	We will show that all but at most $2t$ transpositions end up with weight at least $2-1/t$.
	Clearly, any transposition which is used multiple times ends with weight $2 - 1/t$ as required.
	We will ignore any $\sigma_i,\sigma_{i+1}$ for which some transposition $(1,x)$ with $x \in [t]$ occurs for the first time between them. This disregards at most $2(t-1)$ of the $\sigma_i$.

	We show that all the other $\sigma_i$ (except $\sigma_m$) get weight at least $2-1/t$.
	Consider the section between $\sigma_i = (1,a)$ and $\sigma_{i+1} = (1,b)$
	\[
		(1, a), (1, s_1), \dots, (1, s_\ell), (1, b).
	\]
	If there is a transposition between $\sigma_i$ and $\sigma_{i+1}$ used for neither the first nor the last time, then $\sigma_i$ has weight at least 2.
	So we assume all $(1,s_j)$ are used for either the first or last time.  Let $s_1', \dots, s_{\ell'}'$ be the transpositions used for the last time. If $\ell'\geq t-1$, then $\sigma_i$ receives weight at least $2 - 1/t$ as desired, and we show that this is always the case.

	We show this by proving that we cannot simultaneously leave counters in the positions $\{a, s_1', \dots, s_{\ell'}', b\}$ (which we should be able to do if the set has at most $t$ elements).
    The transposition $\sigma_i = (1,a)$ must be used (as this is the only way to put a counter on $a$) and then the only way for a counter to `reload' position 1 before $\sigma_{i+1}$ is using a transposition $(1, s_{j}')$ that has been used before.
	(Here we use our assumption that no $(1,x)$ appears for the first time in our segment for $x\in [t]$.)
	When we use $(1,s_j')$ to `reload', then no counter can end in position $s_{j}'$ since this is the last use of this transposition. Hence, we cannot `reload' position $1$ to leave a counter in $b$, and we get a contradiction.
\end{proof}

\section{Open problems}
\label{sec:discussion}

We determined exactly the minimum number of transpositions needed in a $2$-reachable network, and gave an upper bound for large values of $t$, but even the asymptotics are still open for large $t$. We also gave lower bounds for $t$-reachable networks using star transpositions, although there is still a gap between the upper and lower bounds even in this restrictive case.

\begin{problem}
What is the minimum number of transpositions needed in a $t$-reachable network? What if all transpositions are of the form $(1,\cdot)$?
\end{problem}

We proved in Theorem \ref{thm:2unifstar} that there is a gap between 2-uniformity and 2-reachability when restricting to star transpositions, although even the correct asymptotics for $2$-uniformity are unknown. It is not hard to construct a 2-uniformity network of length $2n-3$, e.g.~consider the sequence of lazy transpositions
\[
	(1,2, \tfrac{1}{2}),(1,3, \tfrac{2}{n}),(1,2,\tfrac12),(1,4, \tfrac2{n-1}),\dots,(1,2, \tfrac12),(1,n, \tfrac23),(1,2, \tfrac12),
\]
and we conjecture that no smaller sequences exist.\footnote{After a preliminary version of this paper was put on arXiv, Conjecture \ref{cj9} was proven by \cite{janzer2022partial}.}
\begin{conjecture}\label{cj9}
	For $n \geq 2$,  the minimum number of lazy transpositions in a 2-uniformity network is $2n -3$.
\end{conjecture}
We remark that, if this were to be true, it would match nicely with selection networks.

\bibliographystyle{abbrvnat}
\bibliography{short-reachability-networks}
\label{sec:biblio}

\end{document}